\documentclass[11pt,reqno]{amsart}

\usepackage[utf8]{inputenc}
\usepackage[T1]{fontenc}

\usepackage{mathtools}
\mathtoolsset{mathic}
\usepackage{cfr-lm}
\usepackage{dsfont}
\usepackage{microtype}
\usepackage[top=3.75cm, bottom=3cm, left=3cm, right=3cm]{geometry}
\usepackage{bm}
\usepackage[mathscr]{eucal}

\numberwithin{equation}{section}

\usepackage{graphicx}

\usepackage[usenames,dvipsnames,table]{xcolor}

\usepackage[colorlinks=true, citecolor=citation, urlcolor=citation, linkcolor=reference]{hyperref}
\hypersetup{
  pdftitle={Anchored Dyck Paths},
  pdfauthor={Jimmy Dillies},
  colorlinks=true,
  breaklinks=true,
  bookmarksopen=true,
  bookmarksnumbered=true,
  pdfpagemode=UseOutlines,
  plainpages=false,
  linkcolor=Plum,
  citecolor=Plum,
  anchorcolor=green,
  pagebackref=true}
\definecolor{citation}{rgb}{0,.40,.80}
\definecolor{reference}{rgb}{.80,0,.40}

\newcommand{\IN}{\mathbb{N}}
\newcommand{\IZ}{\mathbb{Z}}
\newcommand{\IQ}{\mathbb{Q}}

\theoremstyle{plain}

\newtheorem{thm}{Theorem}

\newtheorem{lem}[thm]{Lemma}
\newtheorem{rmk}[thm]{Remark}
\newtheorem{eg}[thm]{Example}
\newtheorem*{defi}{Definition}

\usepackage{thm-restate}

\usepackage{tikz}
\usetikzlibrary{decorations.markings, arrows.meta}

\tikzset{
  necklace string/.style={gray!55, thin},
  necklace bead/.style={circle, draw=black, inner sep=0pt, minimum size=6.5pt},
  necklace separator/.style={black, line width=.8pt},
  dyck grid/.style={step=1, gray!45, thin},
  dyck diagonal/.style={black!65, thin},
  dyck path/.style={very thick, blue!70!black},
  dyck vertex/.style={blue!70!black}
}
\newcommand{\DrawBinaryNecklace}[2][\NecklaceFigureRadius]{
  \path[use as bounding box] (-1.15,-1.15) rectangle (1.15,1.15);
  \draw[necklace string] (0,0) circle (#1);
  \foreach \ang/\col in {#2}
    \node[necklace bead, fill=\col] at (\ang:#1) {};
}
\newcommand{\DrawDyckPath}[4][\DyckFigureSize]{
  \path[use as bounding box] (0,0) rectangle (#1,#1);
  \pgfmathsetmacro{\dyckscale}{#1/max(#2,#3)}
  \pgfmathsetmacro{\dyckxoffset}{(#1-#2*\dyckscale)/2}
  \pgfmathsetmacro{\dyckyoffset}{(#1-#3*\dyckscale)/2}
  \begin{scope}[shift={(\dyckxoffset,\dyckyoffset)},x=\dyckscale cm,y=\dyckscale cm]
    \draw[dyck grid] (0,0) grid (#2,#3);
    \draw[dyck diagonal] (0,0)--(#2,#3);
    \draw[dyck path] #4;
  \end{scope}
}

\begin{document}

\title{Anchored Dyck Paths}

\author{Jimmy Dillies}
\address{Department of Mathematics, University of Georgia, Athens, GA 30602, USA}
\email{jimmy.dillies@uga.edu}

\date{May, 2026}

\begin{abstract}
We answer a question of Simental by providing a combinatorial interpretation of a formula which generalizes rational Catalan numbers and which appears in the study of  Springer fibers ~\cite{EtingofKrylovEA21}. We provide an interpretation in terms of binary necklaces as well as anchored Dyck paths.
\end{abstract}

\maketitle

\section{Introduction}
\label{sec:intro}

\subsection{Motivation}
``Higher rank'' versions of Catalan numbers appear as a byproduct of representation-theoretic problems and in the study of Springer varieties associated to singular curves.
In~\cite{EtingofKrylovEA21,GonzalezSimentalVazirani24}, the authors encounter the expression
\[
C^{\mathop{gen}}_{m,n}=\frac{\gcd(m,n)}{m+n}\binom{m+n}{n}
\]

This expression generalizes the ordinary rational Catalan numbers and is the specialization at $q=1$ of the normalized rational $q$-Catalan polynomial
\[
\overline C_{m,n}(q)
=
[\gcd(m,n)]_q
\frac{1}{[m+n]_q}
\binom{m+n}{n}_q
\]
from~\cite{XinZhong20}.
Although $C^{\mathop{gen}}_{m,n}$ is always integral, it had no known combinatorial interpretation.
Simental raised the question~\cite{Simental21} of whether it admits an enumerative interpretation, following a problem he learned from Etingof, who attributed it to Stanley.
The aim of this note is to provide two such descriptions: one in terms of necklaces and one in terms of Dyck paths.\\

Note that the connection between rectangular Catalan combinatorics and affine Springer theory goes back to work of Hikita~\cite{Hikita14} and Gorsky-Mazin~\cite{GorskyMazin13}; see also Gorsky-Mazin-Vazirani~\cite{GorskyMazinVazirani16}.
In this setting, the homology of the relevant affine Springer fibers carries an action of the type $A$ rational Cherednik algebra~\cite{OblomkovYun16,GarnerKivinen23,GorskySimentalVazirani24}.
Rational Cherednik algebras of type $A$ are special cases of quantized Gieseker varieties, and the dimension formula of~\cite{EtingofKrylovEA21} is a higher-rank analogue of the rectangular Catalan number.
This representation also has a geometric realization through generalized affine Springer fibers~\cite{GorskySimentalVazirani24}, while~\cite{GonzalezSimentalVazirani24} gives a natural multiparameter deformation of the same dimension.
Thus the enumerative question is natural both from Catalan combinatorics and from affine Springer theory; compare also Conjecture~4 of Xin and Zhong~\cite{XinZhong20}.

\subsection{Catalan numbers}
Any strictly positive rational number $q$ can be written in a unique manner as

\[
q=\frac{a}{b-a}
\]

where $0<a<b$ are coprime integers.
The Catalan number of $q$ is defined as

\[
\mathop{Cat}(q)=\frac{1}{a+b}\binom{a+b}{a}.
\]

If the input is a strictly positive natural number $n$, one recovers the usual Catalan numbers

\[
\mathop{Cat}(n)=\frac{1}{2n+1}\binom{2n+1}{n}=\frac{1}{n+1}\binom{2n}{n}.
\]

Catalan numbers are ubiquitous in algebraic combinatorics - see~\cite{Stanley15}.
In particular, they count the number of Dyck words or Dyck paths of a certain type.
A Dyck word is a valid sequence of parentheses, that is, a string of equally many symbols $``("$ and $``)"$, such that, when read from left to right, the number of closing parentheses never exceeds the number of opening parentheses.
Such a word can also be represented as a lattice path joining the corners of a square grid, which always remains above (or touches) the diagonal and where the only possible moves are \emph{up} [U - $``("$ - $\circ$] and \emph{right} [R - $``)"$ - $\bullet$].
Those paths are also referred to as Dyck paths.
See Figure~\ref{fig:dyck-examples} on the left for a basic example.

\begin{figure}[h!]
\centering
\begin{tikzpicture}
  \begin{scope}
    \DrawDyckPath{3}{3}{(0,0) -- (0,1) -- (1,1) -- (1,2) -- (1,3) -- (2,3) -- (3,3)}
  \end{scope}
  \begin{scope}[xshift=3.5cm]
    \DrawDyckPath{3}{2}{(0,0) -- (0,1) -- (1,1) -- (1,2) -- (2,2) -- (3,2)}
  \end{scope}
\end{tikzpicture}
\caption{Left: the path URUURR corresponding to the Dyck word $()(())$. Right: the rational Dyck path URURR on a $2\times 3$ grid.}
\label{fig:dyck-examples}
\end{figure}
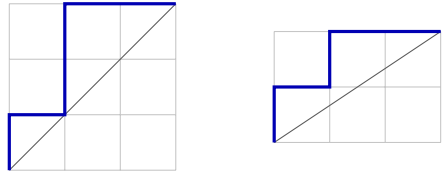

The above construction can be generalized to rational Dyck paths:  lattice paths joining the corners of a rectangular grid, as in Figure~\ref{fig:dyck-examples} on the right.

Dyck paths on an $n\times n$ square are counted by the Catalan numbers

\[
\frac1{n+1}\binom{2n}{n}
\]

while rational $(m,n)$ Dyck paths, i.e. paths on an $m\times n$ rectangle\footnote{Our convention is to first count rows, then columns.},  are counted by

\[
\frac1{m+n}\binom{m+n}{n}
\]

when $m$ and $n$ are coprime.
Notice that the above formulas do not encompass all possible types of Dyck paths: when $m$ and $n$ are distinct but not relatively prime the general formula is significantly more complex and is due to Grossman and Bizley~\cite{Bizley54}.\\

\subsection{Results}

An \emph{anchor} on an $(m,n)$ Dyck path is a vertex whose coordinates lie on the diagonal $ny-mx=0$.
By definition, any Dyck path has at least two anchor points, the initial and final vertex.
Given a Dyck path $p$, we will denote the number of anchor points away from the origin by $a(p)$.

\begin{defi}
The anchored weight of $p$ is

\[
w_P(p)=\frac{{\gcd(m,n)}}{a(p)} \in \IQ
\]

\end{defi}
Note that $w_P(p)$ is in general not an integer as can be seen in Figure~\ref{fig:fullexample}

\begin{restatable}{thm}{dyckpaththeorem}
\label{thm:1}
The function $C^{\mathop{gen}}_{m,n}$ counts the (anchor) weighted number of rational Dyck paths on an $m\times n$ grid.
\end{restatable}

The details of Theorem~\ref{thm:1}  are worked out in Section~\ref{sec:dyck} but are the corollary of another combinatorial description in terms of necklaces:

\begin{restatable}{thm}{necklacetheorem}
\label{thm:2}
The function $C^{\mathop{gen}}_{m,n}$ counts the weighted number of binary necklaces of type $(m,n)$
\end{restatable}

We will recall the notion of necklace in Section~\ref{sec:necklaces} where the weight will also be defined.  The weight roughly encodes the entropy or lack of symmetry of a necklace.
Finally, we will give a fully (unweighted) enumerative interpretation in terms of binary necklaces with some prescribed blocks of length $\frac{m+n}{\gcd(m,n)}$.
A marked necklace will be a necklace together with the choice of a representative block in its orbit under rotational symmetries:

\begin{restatable}{thm}{markednecklacetheorem}
\label{thm:3}
The function $C^{\mathop{gen}}_{m,n}$ counts the number of marked binary necklaces of type $(m,n)$
\end{restatable}

\subsection*{Acknowledgments}
The author would like to thank Jos{é} Simental for helpful suggestions and references, and Andr{é}as Dillies for introducing him to necklaces.

\section{Necklaces}
\label{sec:necklaces}

Recall that a binary necklace of type $(m,n)$ is an equivalence class of strings made of two symbols, the first one appearing $m$ times and the second $n$ times. Combinatorially, it can be described by a closed chain of $m+n$ beads with the above prescribed proportions as is illustrated in Figure~\ref{fig:necklace}.

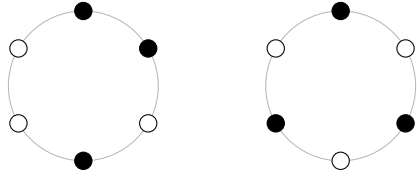
\begin{figure}[h]
\centering
\begin{tikzpicture}[baseline=(current bounding box.center)]
  \path[use as bounding box] (-3.15,-1.35) rectangle (3.15,1.35);
  \begin{scope}[xshift=-1.7cm]
    \DrawBinaryNecklace{90/black,30/black,-30/white,-90/black,-150/white,-210/white}
  \end{scope}
  \begin{scope}[xshift=1.7cm]
    \DrawBinaryNecklace{90/black,30/white,-30/black,-90/white,-150/black,-210/white}
  \end{scope}
\end{tikzpicture}
\caption{Two necklaces of type $(3,3)$.  The left one has weight $w_N=3$ and the right one $w_N=1$.}
\label{fig:necklace}
\end{figure}

\subsection{Weighted bijection.}

Let $\Omega_{m,n}$ be the set of binary necklaces of type $(m,n)$.
Given a particular necklace $\omega \in \Omega_{m,n}$, we denote by $r(\omega)$ (or $r$ when there is no ambiguity) the order of its group of symmetries under rotation.

\begin{defi}
The weight of $\omega$ is
\[
w_N(\omega)=\frac{{\gcd(m,n)}}{r} \in \IN
\]

\end{defi}

Note that $w$ is an integer as a fundamental region for the necklace must contain a multiple of
 $\frac{m+n}{{\gcd(m,n)}}$ beads.
To understand $w$ imagine some necklace $\omega$.
If $\omega$ were to have a rotational symmetry, its fundamental region would have a size multiple of  $\frac{m+n}{{\gcd(m,n)}}$.
Pick a starting bead and partition $\omega$ into subwords of that length.
Then $w$ is the number of \textbf{extrinsically} distinguishable blocks.
Another description of $w_N$ is thus the size (measured in minimal block size) of the fundamental region of a necklace under rotation, it is a certain measure of entropy.\\

As an example, the necklaces of Figure~\ref{fig:necklace} can be broken as follows in blocks of size $\frac63$ :

\begin{figure}[h]
\centering
\begin{tikzpicture}[baseline=(current bounding box.center)]
  \path[use as bounding box] (-3.15,-1.6) rectangle (3.15,1.6);
  \begin{scope}[xshift=-1.7cm]
    \DrawBinaryNecklace{90/black,30/black,-30/white,-90/black,-150/white,-210/white}
    \foreach \ang in {0,-120,120}
      \draw[necklace separator] (\ang:.80) -- (\ang:1.20);
  \end{scope}
  \begin{scope}[xshift=1.7cm]
    \DrawBinaryNecklace{90/black,30/white,-30/black,-90/white,-150/black,-210/white}
    \foreach \ang in {0,-120,120}
      \draw[necklace separator] (\ang:.80) -- (\ang:1.20);
  \end{scope}
\end{tikzpicture}
\caption{Necklace of type $(3,3)$ decomposed into blocks of size $\frac{m+n}{{\gcd(m,n)}}$.}
\label{fig:necklace-blocks}
\end{figure}
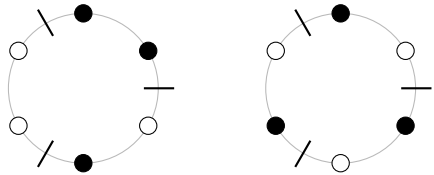

The first one has 3 distinguishable starting blocks, the second only 1; i.e.
\[
w_N(\bullet \circ |  \circ  \bullet | \bullet \circ ) = 3 \quad \textrm{and} \quad
w_N( \bullet \circ |  \bullet \circ | \bullet \circ ) = 1.
\]

We see that the latter corresponds to $\displaystyle \frac{{\gcd(3,3)}}{|\IZ_3|}$ and the former to $\displaystyle \frac{{\gcd(3,3)}}{|\IZ_1|}$.

\necklacetheorem*

\begin{proof}
\begin{equation}
\sum_{\omega \in \Omega_{m,n}} w_N(\omega) =
\sum_{\omega \in \Omega_{m,n}} \frac{{\gcd(m,n)}}{r(\omega)} =
\frac{{\gcd(m,n)}}{m+n} \sum_{\omega \in \Omega_{m,n}} \frac{m+n}{r(\omega)} \label{eq:sumnecklace}
\end{equation}

Write $W_{m,n}$ for the set of binary words of type $(m,n)$, i.e. strings with respectively $m$ and $n$ characters of each type.
Let
\[
\pi: W_{m,n} \rightarrow \Omega_{m,n}
\]

be the obvious quotient map.
The number of elements in the preimage of a necklace is
\[
|\pi^{-1}(\omega)|=\frac{m+n}{r(\omega)}
\]

whence expression~\ref{eq:sumnecklace} becomes
\[
\frac{{\gcd(m,n)}}{m+n} \sum_{\omega \in \Omega_{m,n}} |\pi^{-1}(\omega)| =
\frac{{\gcd(m,n)}}{m+n} \left| W_{m,n} \right| =
\frac{{\gcd(m,n)}}{m+n} \binom{m+n}{n}
\]

\end{proof}

\subsection{Unweighted bijection.}

In the above argument we split our necklace of type $(m,n)$ into segments of length $g=\frac{m+n}{\gcd(m,n)}$.
While the original cut is a priori arbitrary, we will now remove any unnecessary freedom.
Let $s$ be a binary word representing a necklace.
We will order those words lexicographically through the order $\circ \prec \bullet$ (or $U \prec R$).
Let $s(\omega)$ be the lexicographically minimal word representing the necklace $\omega$.
We will now draw a necklace by representing the first bead (symbol) of $s(\omega)$ on top and then proceeding clockwise.\\

A \emph{block} is a succession of $g$ successive beads on a necklace whose counterpart in $s(\omega)$ starts in a position congruent to $1$ modulo $g$.
We say that a block is \emph{distinguishable} if it represents a distinct orbit among these blocks under the rotational symmetries of the necklace\footnote{One could e.g. decide to take the leftmost representative in $s(\omega)$.}.
A \emph{marked} necklace is a necklace together with a distinguishable block.
Figure~\ref{fig:markednecklaces} shows the possible markings on the necklaces URURUR and URUURR.

We can now rephrase Theorem~\ref{thm:2} in this language:

\markednecklacetheorem*

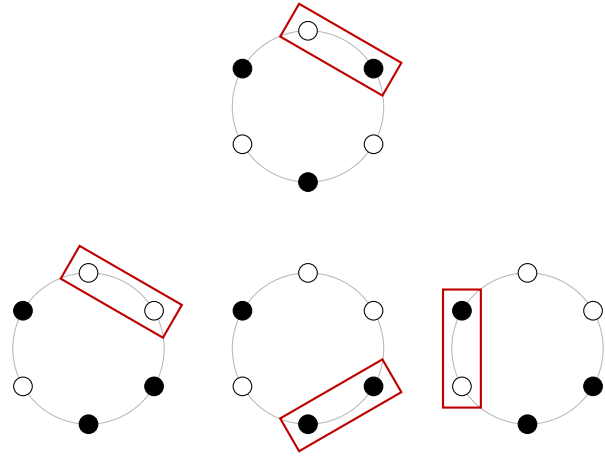
\begin{figure}[h]
\begin{tikzpicture}[
  necklace string/.style={gray!55, thin},
  necklace bead/.style={circle, draw=black, inner sep=0pt, minimum size=7pt},
  necklace marker/.style={red!75!black, line width=.8pt}
]

\newcommand{\MarkedBlockOne}{
  \begin{scope}[shift={(.433,.75)}, rotate=-30]
    \draw[necklace marker] (-.78,-.25) rectangle (.78,.25);
  \end{scope}
}
\newcommand{\MarkedBlockTwo}{
  \begin{scope}[shift={(.433,-.75)}, rotate=30]
    \draw[necklace marker] (-.78,-.25) rectangle (.78,.25);
  \end{scope}
}
\newcommand{\MarkedBlockThree}{
  \begin{scope}[shift={(-.866,0)}, rotate=90]
    \draw[necklace marker] (-.78,-.25) rectangle (.78,.25);
  \end{scope}
}

\newcommand{\NecklaceURURUR}{
  \draw[necklace string] (0,0) circle (1);
  \foreach \ang/\col in {90/white,30/black,-30/white,-90/black,-150/white,150/black}
    \node[necklace bead, fill=\col] at (\ang:1) {};
}

\newcommand{\NecklaceURURRU}{
  \draw[necklace string] (0,0) circle (1);
  \foreach \ang/\col in {90/white,30/white,-30/black,-90/black,-150/white,150/black}
    \node[necklace bead, fill=\col] at (\ang:1) {};
}

\begin{scope}[shift={(0,2.1)}]
  \NecklaceURURUR
  \MarkedBlockOne
\end{scope}

\begin{scope}[shift={(-2.9,-1.1)}]
  \NecklaceURURRU
  \MarkedBlockOne
\end{scope}

\begin{scope}[shift={(0,-1.1)}]
  \NecklaceURURRU
  \MarkedBlockTwo
\end{scope}

\begin{scope}[shift={(2.9,-1.1)}]
  \NecklaceURURRU
  \MarkedBlockThree
\end{scope}

\end{tikzpicture}
\caption{The marked necklaces supported by the necklaces URURUR and URUURR.}
\label{fig:markednecklaces}
\end{figure}

\section{Anchored Dyck Paths}
\label{sec:dyck}

Consider $p$ a rational Dyck path of type $(m,n)$, i.e. with $m$ rows and $n$ columns and encoded by a given Dyck word.

\begin{lem}
\label{lem:a(p)}
There are $a(p)$ (possibly trivial) cyclic translations of this word that give another rational Dyck path.
\end{lem}

\begin{proof}
This can be deduced from the Cycle Lemma~\cite{DvoretzkyMotzkin47,DershowitzZaks90} but let us derive the argument.
Clearly all $a(p)$ diagonal points could serve as starting points.
Necessarily, any other potential starting point of Dyck path must be a valley without which the next vertex would be $(1,0)$, i.e. well under the diagonal.
Take some generic non diagonal valley $P$ and let $D(P)$ be the next diagonal point along the original path.
The slope $\delta$ of the line $PD(P)$ is smaller than $\frac{m}{n}$ or $D(P)$ would lie above the diagonal.
Hence, if we were to start the Dyck path at $P$, the point $D(P)$ would lie under the diagonal, a contradiction.
\end{proof}

\dyckpaththeorem*

\begin{proof}
Let $D_{m,n} \subset W_{m,n}$ be the Dyck words of type $(m,n)$.
We have
\begin{equation}
\sum_{p \in D_{m,n}} w_P(p) =
\sum_{p \in D_{m,n}} \frac{\gcd(m,n)}{a(p)} =
\frac{\gcd(m,n)}{m+n} \sum_{p \in D_{m,n}} \frac{m+n}{a(p)} \label{eq:sumdyck}
\end{equation}

We now count the remaining sum by comparing $W_{m,n}$ with $D_{m,n}$:
through cyclic rotations we have (up to rotational symmetry) a
\[
m+n \longleftrightarrow a(p)
\]

correspondence between Dyck words and generic $(m,n)$ words, whence
\[
 \sum_{p \in D_{m,n}} \frac{m+n}{a(p)} = | W_{m,n}|
 \]

Substituting this into~\ref{eq:sumdyck} we obtain
\[
\sum_{p \in D_{m,n}} w_P(p) =
\frac{\gcd(m,n)}{m+n}\binom{m+n}{n} =
C^{\mathop{gen}}_{m,n}.
\]

\end{proof}

\begin{rmk}
Needless to say, we could have arrived at the same result by providing a correspondence between $D_{m,n}$ and $\Omega_{m,n}$.
\end{rmk}

\begin{eg}
Figure~\ref{fig:fullexample} displays all necklaces, their marked necklaces (rotated to be lexicographically minimal when read from the top), and Dyck paths of type $(3,3)$ with their corresponding weights. As expected, in each case the weighted sum adds up to $C^{\mathop{gen}}_{3,3}=10$.
\end{eg}

\begin{figure}[h]
\centering
\begin{tikzpicture}
\path[use as bounding box] (-1.5,-11.25) rectangle (11.4,2.15);
\node[font=\bfseries\tiny] at (0,1.75) {Necklace};
\node[font=\bfseries\tiny] at (3.65,1.75) {Marked necklace(s)};
\node[font=\bfseries\tiny] at (8.65,1.75) {Dyck path representative(s)};

\newcommand{\exampleNecklace}[3]{
  \begin{scope}[shift={(0,#1)}]
    \DrawBinaryNecklace{#2}
    \node[font=\scriptsize] at (0,-1.38) {$w_N=#3$};
  \end{scope}
}
\newcommand{\exampleDyck}[3]{
  \begin{scope}[shift={(#1,#2)}]
    \DrawDyckPath{3}{3}{#3}
  \end{scope}
}
\newcommand{\exampleMarkedBlockOne}{
  \begin{scope}[shift={(.433,.75)}, rotate=-30]
    \draw[red!75!black, line width=.8pt] (-.78,-.25) rectangle (.78,.25);
  \end{scope}
}
\newcommand{\exampleMarkedBlockTwo}{
  \begin{scope}[shift={(.433,-.75)}, rotate=30]
    \draw[red!75!black, line width=.8pt] (-.78,-.25) rectangle (.78,.25);
  \end{scope}
}
\newcommand{\exampleMarkedBlockThree}{
  \begin{scope}[shift={(-.866,0)}, rotate=90]
    \draw[red!75!black, line width=.8pt] (-.78,-.25) rectangle (.78,.25);
  \end{scope}
}
\newcommand{\exampleMarkedNecklaceOne}[2]{
  \begin{scope}[shift={(3.65,#1)}, scale=.48, transform shape]
    \DrawBinaryNecklace{#2}
    \exampleMarkedBlockOne
  \end{scope}
}
\newcommand{\exampleMarkedNecklacesThree}[2]{
  \begin{scope}[shift={(2.45,#1)}, scale=.48, transform shape]
    \DrawBinaryNecklace{#2}
    \exampleMarkedBlockOne
  \end{scope}
  \begin{scope}[shift={(3.65,#1)}, scale=.48, transform shape]
    \DrawBinaryNecklace{#2}
    \exampleMarkedBlockTwo
  \end{scope}
  \begin{scope}[shift={(4.85,#1)}, scale=.48, transform shape]
    \DrawBinaryNecklace{#2}
    \exampleMarkedBlockThree
  \end{scope}
}
\newcommand{\exampleNDyckWeight}[3]{
  \node[font=\scriptsize] at (#1,#2) {$w_N=#3$};
}
\newcommand{\examplePDyckWeight}[3]{
  \node[font=\scriptsize] at (#1,#2) {$w_P=#3$};
}

\exampleNecklace{0}{90/black,30/black,-30/black,-90/white,-150/white,-210/white}{3}
\exampleMarkedNecklacesThree{0}{90/white,30/white,-30/white,-90/black,-150/black,-210/black}
\exampleDyck{7.55}{-1.1}{(0,0)--(0,1)--(0,2)--(0,3)--(1,3)--(2,3)--(3,3)}
\examplePDyckWeight{8.65}{-1.5}{3}

\exampleNecklace{-3}{90/black,30/black,-30/white,-90/black,-150/white,-210/white}{3}
\exampleMarkedNecklacesThree{-3}{90/white,30/white,-30/black,-90/black,-150/white,-210/black}
\exampleDyck{6.25}{-4.1}{(0,0)--(0,1)--(1,1)--(1,2)--(1,3)--(2,3)--(3,3)}
\examplePDyckWeight{7.35}{-4.5}{\frac32}
\exampleDyck{8.95}{-4.1}{(0,0)--(0,1)--(0,2)--(1,2)--(2,2)--(2,3)--(3,3)}
\examplePDyckWeight{10.05}{-4.5}{\frac32}

\exampleNecklace{-6}{90/black,30/black,-30/white,-90/white,-150/black,-210/white}{3}
\exampleMarkedNecklacesThree{-6}{90/white,30/white,-30/black,-90/white,-150/black,-210/black}
\exampleDyck{7.55}{-7.1}{(0,0)--(0,1)--(0,2)--(1,2)--(1,3)--(2,3)--(3,3)}
\examplePDyckWeight{8.65}{-7.5}{3}

\exampleNecklace{-9}{90/black,30/white,-30/black,-90/white,-150/black,-210/white}{1}
\exampleMarkedNecklaceOne{-9}{90/white,30/black,-30/white,-90/black,-150/white,-210/black}
\exampleDyck{7.55}{-10.1}{(0,0)--(0,1)--(1,1)--(1,2)--(2,2)--(2,3)--(3,3)}
\examplePDyckWeight{8.65}{-10.5}{1}
\end{tikzpicture}

\caption{$C^{\mathop{gen}}_{3,3}=10$}
\label{fig:fullexample}
\end{figure}


\begin{thebibliography}{GSV24b}

\bibitem[Biz54]{Bizley54}
Michael Terence~Lewis Bizley.
\newblock Derivation of a new formula for the number of minimal lattice paths
  from (0, 0) to (km, kn) having just t contacts with the line my= nx and
  having no points above this line; and a proof of {G}rossman's formula for the
  number of paths which may touch but do not rise above this line.
\newblock {\em Journal of the Institute of Actuaries}, 80(1):55--62, 1954.

\bibitem[DM47]{DvoretzkyMotzkin47}
Aryeh Dvoretzky and Theodore Motzkin.
\newblock A problem of arrangements.
\newblock {\em Duke Mathematical Journal}, 14(2):305--313, 1947.

\bibitem[DZ90]{DershowitzZaks90}
Nachum Dershowitz and Shmuel Zaks.
\newblock The cycle lemma and some applications.
\newblock {\em European Journal of Combinatorics}, 11(1):35--40, 1990.

\bibitem[EKLS21]{EtingofKrylovEA21}
Pavel Etingof, Vasily Krylov, Ivan Losev, and Jos{\'e} Simental.
\newblock Representations with minimal support for quantized {G}ieseker
  varieties.
\newblock {\em Mathematische Zeitschrift}, 298(3):1593--1621, 2021.

\bibitem[GK23]{GarnerKivinen23}
Niklas Garner and Oscar Kivinen.
\newblock Generalized affine {S}pringer theory and {H}ilbert schemes on planar
  curves.
\newblock {\em International Mathematics Research Notices}, 2023(8):6402--6460,
  2023.

\bibitem[GM13]{GorskyMazin13}
Eugene Gorsky and Mikhail Mazin.
\newblock Compactified {J}acobians and {$q,t$}-{C}atalan numbers, {I}.
\newblock {\em Journal of Combinatorial Theory, Series A}, 120(1):49--63, 2013.

\bibitem[GMV16]{GorskyMazinVazirani16}
Eugene Gorsky, Mikhail Mazin, and Monica Vazirani.
\newblock Affine permutations and rational slope parking functions.
\newblock {\em Transactions of the American Mathematical Society},
  368(12):8403--8445, 2016.
\newblock \href {https://arxiv.org/abs/1403.0303} {\path{arXiv:1403.0303}}.

\bibitem[GSV24a]{GonzalezSimentalVazirani24}
Nicolle Gonz{\'a}lez, Jos{\'e} Simental, and Monica Vazirani.
\newblock Higher rank (q,t)-{C}atalan polynomials, affine springer fibers, and
  a finite rational shuffle theorem.
\newblock {\em Transactions of the American Mathematical Society},
  377(07):5087--5127, 2024.

\bibitem[GSV24b]{GorskySimentalVazirani24}
Eugene Gorsky, Jos{\'e} Simental, and Monica Vazirani.
\newblock From representations of the rational {C}herednik algebra to parabolic
  {H}ilbert schemes via the {D}unkl--{O}pdam subalgebra.
\newblock {\em Transformation Groups}, 29(2):647--716, 2024.
\newblock \href {https://arxiv.org/abs/2004.14873} {\path{arXiv:2004.14873}}.

\bibitem[Hik14]{Hikita14}
Tatsuyuki Hikita.
\newblock Affine {S}pringer fibers of type {A} and combinatorics of diagonal
  coinvariants.
\newblock {\em Advances in Mathematics}, 263:88--122, 2014.
\newblock \href {https://arxiv.org/abs/1203.5878} {\path{arXiv:1203.5878}}.

\bibitem[OY16]{OblomkovYun16}
Alexei Oblomkov and Zhiwei Yun.
\newblock Geometric representations of graded and rational {C}herednik
  algebras.
\newblock {\em Advances in Mathematics}, 292:601--706, 2016.
\newblock \href {https://arxiv.org/abs/1407.5685} {\path{arXiv:1407.5685}}.

\bibitem[Sim21]{Simental21}
Jos{\'e} Simental.
\newblock Higher {C}atalan combinatorics: geometry and representation theory.
\newblock Algebra Seminar, University of Georgia, December 2021.

\bibitem[Sta15]{Stanley15}
Richard~P Stanley.
\newblock {\em Catalan numbers}.
\newblock Cambridge University Press, 2015.

\bibitem[XZ20]{XinZhong20}
Guoce Xin and Yueming Zhong.
\newblock On parity unimodality of {$q$}-catalan polynomials.
\newblock {\em The Electronic Journal of Combinatorics}, 27(1):Paper No. 1.3,
  2020.
\newblock \href {https://arxiv.org/abs/1912.01829} {\path{arXiv:1912.01829}}.

\end{thebibliography}
\end{document}